\documentclass[a4paper, 12pt]{amsart}

\usepackage[utf8x]{inputenc}
\usepackage[T1,T2A]{fontenc}
\usepackage[english]{babel}
\usepackage{amsmath}
\usepackage{amsthm}
\usepackage{abstract}
\usepackage{amsfonts}
\usepackage{graphicx}
\usepackage[in]{fullpage}
\usepackage[colorinlistoftodos]{todonotes}

\newtheorem{theorem}{Theorem}

\newtheorem*{unsignedlem}{Lemma}
\newtheorem*{propos}{Proposition}

\newtheorem*{problem}{Gap}

\theoremstyle{definition}
\newtheorem*{definition}{Definition}

\newtheorem*{remark}{Remark}
\newtheorem*{example}{Example}

\title{Betti numbers of small covers and their two-fold coverings}
\date{}
\author{Dmitry Ulyumdzhiev}
\thanks{The work is partially supported by the grant of the President of the Russian Federation MD - 2907.2017.1}

\address{Moscow State University}
\address{Steklov Mathematical Institute of RAS}
\email{dulumzhiev@gmail.com}

\begin{document}
\maketitle

\begin{abstract}
We point out a gap in the proof of the Davis--Januszkiewicz theorem on cohomology of small covers of simple polytopes, and give a correction to this proof. We use this theorem to  compute explicitly the Betti numbers for a wide class of two-fold coverings over small covers. We describe a series of examples in this class and explain why this class turns out to be special in the context of the general problem of computing the modulo two Betti numbers  of two-fold coverings over small covers.
\end{abstract}

\section{Introduction}

This paper contains two main subjects.

First, we describe a gap in the proof of the famous theorem on cohomology of small covers due to M. W. Davis and T. Januszkiewicz. We present a short correction to the original proof. 

Second, we study the problem of computation of the modulo two Betti numbers of two-fold coverings over small covers of simple polytopes. It is clear that this computation can be reduced to the computation of dimensions of certain special $\mathbb{Z}_2$-algebras, e.g. using the Gysin exact sequence. However, this approach does not allow to obtain an explicit answer for almost all interesting series of two-fold coverings over small covers. We point out a wide class of two-fold coverings over small covers for which the problem of computation of the modulo two Betti numbers  can be solved explicitly. Namely, we introduce the concept of \textit{section classes} in the first cohomology of a small cover. From various points of view these classes turn out to be the simplest ones. If a two-fold covering over a small cover corresponds to a section class, we obtain an explicit combinatorial formula for its modulo two Betti numbers. 

The study of the topology of two-fold coverings over small covers of simple polytopes is motivated among others by connections with the classical problem on realization of cycles, cf.~\cite{Gaif}. Recall that an oriented connected closed smooth manifold $M^n$ is said to be a \textit{Universal Manifold for Realization of Cycles} (or simply a \textit{URC-manifold}) if for any topological space $X$ and any homology class $z \in H_n (X ; \mathbb{Z})$, there exists a finite-fold covering $\widehat{M} \to M^n$ and a continuous mapping $f \colon \widehat{M} \to X$ such that $f_{*} [\widehat{M}] = kz$ for a non-zero integer $k$. In~\cite{Gaif}, Gaifullin found many examples of URC-manifolds among  the orientation two-fold coverings over small covers of simple polytopes, in particular, of the so-called graph-associahedra. Hence, it is natural to ask which of these manifolds is the `smallest' in any reasonable sense. For instance, one may ask which of these manifolds has the smallest modulo two Betti numbers. Unfortunately, the author still cannot answer this question, since for most of small covers over graph-associahedra, their orientation two-fold coverings correspond to non-sectional cohomology classes.

\section{Small covers}

Small covers of simple convex polytopes were originally introduced by Davis and Januszkiewicz in their pioneer work~\cite{DJ1}. Recall the definition of small covers and their basic properties.

Let $P$ be an $n$-dimensional simple convex  polytope with facets $F_1, \dots, F_m$. A map $$\lambda \colon \{ F_1,\ldots,F_m \} \to \mathbb{Z}_2^n$$ is called \textit{characteristic} if for each pairwise different facets $F_{i_1}, \dots, F_{i_n}$ intersecting in a vertex, the vectors $\lambda (F_{i_1}),\ldots,\lambda(F_{i_n})$ form a basis of $\mathbb{Z}_2^n$. For a point $x\in P$, we denote by $\mathrm{St}_{\lambda} (x)$ the linear span of the vectors~$\lambda (F)$, where $F$ runs over  all facets containing~$x$. We denote by~$\lambda_i(F)$ the $i$th coordinate of~$\lambda (F)$.

\begin{definition}
The \textit{small cover $M_{P, \lambda}$} of a simple polytope $P$ corresponding to a characteristic map $\lambda$ is the quotient space of $P \times \mathbb{Z}_2^n$ by the equivalence relation~$\sim$ such that  $(x, a) \sim (y, b)$ if and only if $$\begin{cases} x = y \\ (a + b) \in \mathrm{St}_{\lambda} (x) \mbox{.} \end{cases}$$
\end{definition}

Denote by $[x, a]$ the equivalence class of a point $(x, a) \in P \times \mathbb{Z}_2^n$.
The small cover $M_{P, \lambda}$ is an $n$-dimensional smooth manifold  with the natural $\mathbb{Z}_2^n$-action given by $b \cdot [x, a] = [x, a + b]$. The orbit space of this action is~$P$, the projection map $p \colon M_{P, \lambda} \to P$ is given by $[x, a] \mapsto x$, and the stabilizer subgroup of a point $x\in P$ is $\mathrm{St}_{\lambda} (x)$.  


\section{Davis--Januszkiewicz theorems}

Recall that the $h$-numbers of a simple polytope~$P$ are the numbers $h_i=h_i(P)$, $i=1,\ldots,n$, given by
$$
h_0t^n+h_1t^{n-1}+\cdots+h_n=(t-1)^n+f_0(t-1)^{n-1}+\cdots+f_{n-1},
$$
where $f_i$ is the number of codimension $i+1$ faces of~$P$.

\begin{theorem}[\cite{DJ1}, Theorem 3.1, Corollary 3.7]\label{Th1}

The modulo $2$ Betti numbers of $M_{P, \lambda}$ are equal to the $h$-numbers of~$P$. The action of $\mathbb{Z}_2^n$ on $H^{*} (M_{P, \lambda} ; \mathbb{Z}_2)$  induced by the natural  action of~$\mathbb{Z}_2^n$ on~$M_{P,\lambda}$ is trivial.

\end{theorem}

Theorem \ref{Th1} plays a key role in the proof of the following classical theorem.

\begin{theorem}[\cite{DJ1}, Theorem 4.14, Proposition 3.10]\label{Th2}
There is an  isomorphism of graded rings
$$H^{*}(M_{P, \lambda}; \mathbb{Z}_2) \cong \mathbb{Z}_2[v_1, \ldots, v_m] / (\mathcal{I} + \mathcal{J}), $$
where $\mathcal{I}$ is the Stanley--Reisner ideal of~$P$ and $\mathcal{J}$ is the ideal generated by the linear forms $\sum_{k = 1}^{m} \lambda_i (F_k)v_k$, $i = 1, \ldots, n$. The generators $v_i$ are the cohomology classes Poincar\'e dual to the pre-images $p^{-1}(F_i)$ of the facets~$F_i$ of~$P$, respectively.
\end{theorem}

A key role in the proof of Theorem~\ref{Th1} due to Davis and Januszkiewicz is played by the following construction (cf.~\cite{DJ1}, Section~3).

Let $\ell(x) = \langle x, l \rangle$ be a generic height function on  $\mathbb{R}^n$, i.e., a height function corresponding to a vector~$l$ that is not orthogonal any edge of~$P$. We shall say that a set $A$ lies \textit{below} a set $B$ if $\ell(x) \leq \ell(y)$ for all $x \in A$, $y \in B$. The \textit{index} of a vertex $v$ of $P$ is a number of vertices  adjacent to~$v$ that lie below $v$.

Let $F_v$ be the union of the relative interiors of all faces of~$P$ containing $v$ and lying below $v$. The closure of~$F_v$ is a face of~$P$, and $p^{-1} (F_v)$ is an open cell of dimension equal to the index of~$v$. 
The number of $i$-dimensional cells~$p^{-1} (F_v)$ is equal to the number of index~$i$ vertices of~$P$. It is well known that the latter number is independent of~$\ell$  and is equal to~$h_i (P)$, see~\cite{Br3}. 

Davis and Januszkiewicz claimed and used substantially in their proof of Theorem~\ref{Th1} that the cells~$p^{-1} (F_v)$  form a cell structure on~$M_{P, \lambda}$. Nevertheless, this is not always true.

\begin{problem}
The boundary of a cell $p^{-1} (F_v)$ can be not contained in the union of cells~$p^{-1} (F_w)$ of smaller dimensions. Hence,  the cells~$p^{-1} (F_v)$ generally do not form a CW decomposition of~$M_{P, \lambda}$.
\end{problem}

\begin{example}
Consider a pentagon $ABCDE$ with an arbitrary characteristic map~$\lambda$. Assume that the height function~$\ell$ is chosen so that $\ell(A)>\ell(B)>\ell(C)>\ell(D)>\ell(E)$. Then both endpoints of the one-dimensional cell~$p^{-1}(F_B)$ coincide with the point~$p^{-1}(C)$, which is  contained in the relative interior of the one-dimensional cell~$p^{-1} (F_C)$.
\end{example}

\section{Correction of the proof}

In this section we describe how to fill  in the gap in the proof of Theorem~\ref{Th1}.

Let $s_i$ be the sequence of heights of vertices of~$P$ numbered in the ascending order. Consider the filtration on $M_{P, \lambda}$ formed by the sets 
$$G_p := \bigcup_{\ell(v) \leq s_p} F_v,$$ 
and the corresponding spectral sequence in cohomology with the coefficient group~$\mathbb{Z}_2$. The first page of this spectral sequence is 
$$ E_1^{p, q} = H^{p + q} (G_p, G_{p - 1}; \mathbb{Z}_2).$$
Since the filtration~$\{G_p\}$ is invariant under the action of~$\mathbb{Z}_2^n$ on~$M_{P, \lambda}$, we obtain a well-defined $\mathbb{Z}_2^n$-action on the spectral sequence~$E^r_{p,q}$.
 Since for any vertex~$v$ the set $p^{-1}(F_v)$ is an open cell, the quotient space $G_p / G_{p - 1}$ is homeomorphic to the sphere of dimension equal to the index of the vertex at height~$s_p$. Hence every group~$ E_1^{p, q}$ is either trivial or isomorphic to~$\mathbb{Z}_2$. Therefore, the action of~$\mathbb{Z}_2^n$ on every~$E_1^{p,q}$, hence, on every~$E_r^{p,q}$ for $r>1$ is trivial. Thus, $\mathbb{Z}_2^n$ acts trivially on the modulo~$2$ cohomology of $M_{P, \lambda}$. We have: 
\begin{equation*}
 \dim H^{i}(M_{P, \lambda}; \mathbb{Z}_2) = \sum_{p + q = i}\dim E_{\infty}^{p, q} \leq \sum_{p + q = i}\dim E_1^{p, q} = h_i(P).
 \end{equation*} 
These estimates and the triviality of the $\mathbb{Z}_2$-action on~$H^*(M_{P, \lambda};\mathbb{Z}_2)$  are exactly those two facts that were deduced by Davis and Januszkiewicz from the wrong fact that the cells~$p^{-1} (F_v)$ form a CW decomposition of~$M_{P, \lambda}$. Since we have given another proof of these two fact, we can further proceed with the original proof of Theorem~\ref{Th1} and then Theorem~\ref{Th2} due to Davis and Januszkiewicz without any changes. As a consequence, the estimates above turn out to be equalities:
$$ \dim H^{i}(M_{P, \lambda}; \mathbb{Z}_2)  = h_i (P).$$

\begin{remark}


There are analogs of Theorems~\ref{Th1} and~\ref{Th2} for quasi-toric manifolds, which are also due to Davis and Januszkiewicz. There is a similar gap in their proofs, and this gap can be filled in similarly. However, unlike Theorems~\ref{Th1} and~\ref{Th2}, their analogs for quasi-toric manifolds admit a different  proof obtained by Buchstaber and Panov, see e.g.~\cite{BP2}.

An analog of the spectral sequence of the filtration $\{ G_p \}$ can be constructed for any covering over~$M_{P, \lambda}$. However, generally it will have many non-trivial differentials.
\end{remark}

\section{Two-fold coverings over small covers: Section classes}

Let $S$ be a hyperplane section of $P$. Its pre-image~$p^{-1}(S)$  is a modulo~$2$ singular cycle in~$M_{P,\lambda}$, but if $S$ is not a facet of $P$, then this cycle is homologically trivial. On the other hand, if $p^{-1} (S)$ is disconnected then its connected component always represents a non-trivial homology class.

Pre-image of $S$ may be non-connected only if $S$ is a generic section, i.\,e., contains no vertex of~$P$. Further we assume that either $S$ is a generic hyperplane section with disconnected pre-image~$p^{-1}(S)$ or $S$ is a facet of~$P$. In the former case we denote by $M_{S, \lambda_S}$ a connected component of $p^{-1} (S)$, and in the latter case we put  $M_{S, \lambda_S}=p^{-1} (S)$. It is easy to see that in both cases $M_{S, \lambda_S}$ is the small cover over $S$ corresponding to the characteristic function given by $\lambda_S(F\cap S)=\lambda(F)$ whenever $F\cap S$ is non-empty and does not coincide with~$S$.

\begin{definition}
We say that a class $w \in H^{1} (M_{P, \lambda}; \mathbb{Z}_2)$ is a \textit{section class} if there exists an $S$ as above such that $w$ is the Poincar\'e dual of the homology class of~$M_{S, \lambda_S}$.
\end{definition}

By definition, the section classes corresponding to the facets of $P$ are the generators~$v_i$. If $S$ is a generic hyperplane section with disconnected $p^{-1}(S)$, then $(D[M_{S, \lambda_S}])^2 = 0$, where $D$ stands for the Poincar\'e duality operator. Hence, non-triviality of~$w^2$ is an obstruction for~$w$ to be a section class corresponding to a generic hyperplane section. 

It is interesting if the orientation two-fold covering over $M_{P, \lambda}$ can correspond to a section class. Let us consider a simple example.

\begin{example}
Suppose that $P$ admits a regular coloring of its facets in $n$ colors. Then $P$ admits a characteristic map~$\lambda$ that maps every facet of the $i$th color to the $i$th basic vector~$e_i \in \mathbb{Z}_2^n$. For instance, a permutohedron admits such coloring. Choose an element $a \in \langle e_2, \dots, e_n \rangle$ and a facet $G$ of the first color, and consider the map $\nu$ given by 
$$ \nu(F) = 
\left\{
\begin{aligned} 
&\lambda(G) + a  &&\text{if }F = G, \\ &\lambda(F) &&\text{if }F \neq G.
\end{aligned}
\right.
$$ One can easily check that $\nu$ is characteristic. Now, suppose that $a$ has an odd number of non-zero coordinates.  Then the first Stiefel--Whitney class of $M_{P, \nu}$ is the Poincar\'e dual of~$p^{-1}(G)$. Hence, the orientation two-fold covering over~$M_{P, \nu}$ corresponds to a section class.
\end{example}

\section{The Betti numbers of two-fold coverings corresponding to section classes}

It is well known that equivalence classes of two-fold coverings over a topological space are in one-to-one correspondence with elements of its first cohomology group with the coefficient group~$\mathbb{Z}_2$. We denote by~$M_w$ the two-fold covering over~$M_{P, \lambda}$ corresponding to a class $w \in H^{1} (M_{P, \lambda}; \mathbb{Z}_2)$.

We start with the following proposition, which shows that the problem of the computation of the Betti numbers of $M_w$ can be reduced to the case when $w$ is a sum of two section classes corresponding to facets. Thus it shows that section classes are reasonable family of cohomology classes in the context of this problem.

\begin{propos}
The problem of computing the Betti numbers of $M_w$ can be reduced to the case of $w = D[p^{-1} (F)] + D[p^{-1} (G)]$, where $F$ and $G$ are facets of a polytope.
\end{propos}

\begin{proof}
Suppose that $w = \sum_{i \in I} D[p^{-1}(F_i)]$ for some set of indices $I$. It is easy to see that the following map on the set of facets of the prism $P \times [0, 1]$ is characteristic: 
$$ 
\lambda_w (G) = \left\{
\begin{aligned}
        &\lambda(F_i) &&\text{if }G=F_i \times [0, 1]\text{ and }i \notin I,\\
        &\lambda(F_i) + e_{n + 1} &&\text{if }G=F_i \times [0, 1]\text{ and }i \in I,\\
        &e_{n + 1} &&\text{if }G = P \times \{ 0 \}\text{ or }P \times \{ 1 \}.
        \end{aligned} 
        \right.
        $$
It is not hard to check that $M_{P \times [0, 1], \lambda_w}$ coincides with the projectivization~$P(\xi\oplus 1)$, where $\xi$ is the real linear bundle over $M_{P, \lambda}$ corresponding to~$w$, and $1$ denotes the trivial  real linear bundle over $M_{P, \lambda}$. Moreover, if $\pi \colon M_{P \times [0, 1], \lambda_w} \to M_{P, \lambda}$ is the projection map, then $\pi^{*}(w) = D[\tilde{p}^{-1}(P \times 0)] + D[\tilde{p}^{-1}(P \times 1)]$, where $\tilde{p} \colon M_{P \times [0, 1], \lambda_w} \to P \times [0, 1]$ is the projection of the small cover $M_{P \times [0, 1], \lambda_w}$.

Finally, the two-fold covering over $M_{P \times [0, 1], \lambda_w}$ corresponding to $\pi^{*} (w)$ is homeomorphic to $M_w \times \mathbb{S}^1$. Consequently, if one could compute the Betti numbers of this covering space, then he would immediately  obtain the Betti numbers of the original two-fold covering~$M_w$.
\end{proof}

Let $h^{*} (-; \mathbb{Z}_2)$ be the dimension of $H^{*} (- ; \mathbb{Z}_2)$. The main result of this section is as follows.

\begin{theorem}\label{Th3}
Suppose that $ w \in H^1 (M_{P, \lambda}; \mathbb {Z}_2) $ is a section class corresponding to~$S$. Then $$ h^{m} (M_w; \mathbb{Z} _2) = 2h_{m} (P) - h_{m-1} (S) - h_{m} (S),$$ 
where $ h_i(S) $ and $ h_i(P) $ are the $h$-numbers of~$S$ and~$P$, respectively.
\end{theorem}

The proof of this theorem reduces to the following lemma.

\begin{unsignedlem}
Let $ X $ be a connected smooth closed manifold and let $ \widetilde {X} $ be the two-fold covering of it corresponding to a non-trivial class $ w \in H^1 (X; \mathbb {Z} _2) $. Let $ i \colon Y \hookrightarrow X$ be a connected smooth submanifold of $ X $ such that $ i_{*} [Y] $ is the Poincar\'e dual of~$ w $. Then $$ h^{m} (\tilde {X}; \mathbb {Z} _2) = 2h ^ {m} (X; \mathbb {Z} _2) - h ^ {m} (Y; \mathbb{Z}_2) - h^{m-1} (Y; \mathbb {Z}_2). $$
\end{unsignedlem}

\begin{proof}

Recall the exact Gysin sequence: $$\ldots \to H ^ {m} (X; \mathbb{Z} _2) \to H ^ m (\widetilde {X}; \mathbb{Z}_2) \to H ^ m (X; \mathbb {Z} _2) \xrightarrow{w \smile -}  H ^ {m + 1} (X; \mathbb{Z}_2) \to \ldots $$ 
The exactness of this sequence easily implies that $$ h ^ {m + 1} (\widetilde {X}; \mathbb {Z} _2) = h ^ {m + 1} (X; \mathbb {Z} _2) - h ^ {m} (X; \mathbb{Z}_2) + k_m + k_{m + 1}, $$ where $ k_m $ is the dimension of the kernel of $ (w \smile -) \colon H ^ {m} (X; \mathbb {Z} _2) \to H ^ {m + 1} (X; \mathbb {Z} _2) $.
Thus, the proof reduces to computing the numbers $ k_m $. The manifolds $X$ and~$Y$ are closed and connected, hence, for any cohomology class $ \sigma \in H ^ {*} (X; \mathbb{Z}_2) $, there is a homology class $ z \in H_ {n - 1 - *} (Y; \mathbb { Z} _2) $ such that
$$\begin{cases} D_X (w \smile \sigma) = i_* (z), \\ D_Y (z) = i ^ * (\sigma), \end{cases}$$
where $D_X$ and~$D_Y$ are the Poincar\'e duality operators for~$X$ and~$Y$, respectively.
Since $ i_{*} ([Y]) \neq 0 $, the map $ i_ {*} $ is injective, hence, $ \mathrm {Ker} (w \smile -) = \mathop{\mathrm {Ker}} i ^ *$. Similarly, the dual map $i^*$ is surjective,  therefore, $ k_m = h ^ {m} (X; \mathbb {Z} _2) - h ^ {m} (Y; \mathbb {Z} _2). $
\end{proof}

\begin{proof}[Proof of Theorem~\ref{Th3}]
By definition, a section class cannot be trivial. Since  the modulo two  Betti numbers of a small cover of a polyhedron coincide with the $h$-numbers of this polyhedron, the theorem follows from the above lemma.
\end{proof}

\section*{Acknowledgements} 

The author would like to thank his advisor Alexander A. Gaifullin for his constant guidance and advice throughout this research.

\end{document}